\documentclass[10pt]{amsart}

\usepackage{a4}
\usepackage{amssymb}
\usepackage{amscd}

\newtheorem{thm}{Theorem}[section]

\newtheorem{cor}[thm]{Corollary}

\theoremstyle{definition}

\newtheorem{rmk}[thm]{Remark}
\newtheorem{exa}[thm]{Example}

\def\quod{\hskip 0.5em\relax }

\newcommand{\Z}{\mathbf{Z}}
\newcommand{\Q}{\mathbf{Q}}
\newcommand{\R}{\mathbf{R}}

\newcommand{\Mm}[2]{\mathcal M_{{#1},{#2}}}
\newcommand{\Mmb}[2]{\overline{\mathcal M}_{{#1},{#2}}}
\newcommand{\Mmd}[2]{\mathcal M^{\delta}_{{#1},{#2}}}

\begin{document}
\pagestyle{plain}

\title{Inversion of series and the cohomology of the moduli spaces $\Mmd{0}{n}$.}

\author{Jonas Bergstr\"om}
\email{o.l.j.bergstrom@uva.nl}
\address{Korteweg-de Vries Instituut, Universiteit van Amsterdam, Plantage \newline Muidergracht 24, 1018 TV Amsterdam, The Netherlands.}
\author{Francis Brown}
\email{brown@math.jussieu.fr}
\address{CNRS and  Institut Math\'ematiques de Jussieu. \newline
 175 rue du Chevaleret,  75013 Paris, France.}

\maketitle


\begin{abstract}
For $n\geq 3$,  let $\Mm{0}{n}$ denote the moduli space of genus 0
curves with $n$ marked points, and $\Mmb{0}{n}$ its smooth
compactification. A theorem due to Ginzburg, Kapranov and Getzler states that
the
 inverse of the  exponential generating series for the Poincar\'e polynomial of
$H^{\bullet}(\Mm{0}{n})$ is given by the corresponding series for
$H^{\bullet}(\Mmb{0}{n})$. In this paper, we prove that the inverse
of the  ordinary generating series for the Poincar\'e polynomial of
$H^{\bullet}(\Mm{0}{n})$ is given by the corresponding series for
$H^{\bullet}(\Mmd{0}{n})$, where $\Mm{0}{n}\subset \Mmd{0}{n}
\subset \Mmb{0}{n}$ is a certain smooth affine scheme.  
\end{abstract}

\section{Introduction}
For $n\geq 3$, let  $\Mm{0}{n}$ be the moduli space, defined over
$\Z$,  of smooth $n$-pointed curves of genus zero, and let
$\Mm{0}{n}\subset \Mmb{0}{n}$ denote  its smooth compactification,
due to Deligne-Mumford and Knudsen. In \cite{Br}, an intermediary
space $\Mmd{0}{n}$, which satisfies
$$\Mm{0}{n}\subset \Mmd{0}{n}
\subset \Mmb{0}{n},$$ 
was defined 
in terms  of explicit polynomial
equations. It is a smooth affine scheme over $\Z$.
The automorphism group of $\Mm{0}{n}$ is the symmetric group
$\mathfrak{S}_n$ permuting the $n$ marked points, and this gives
rise to a decomposition (see \cite{Br}),
$$\Mmb{0}{n}=\bigcup_{\sigma\in \mathfrak{S}_n} \sigma(\Mmd{0}{n})\ .$$
Thus $\Mmd{0}{n}$ defines a symmetric set of canonical affine charts
for  $\Mmb{0}{n}$.

In this note, we compute the dimensions
$$a_{n,i}:=\dim_\Q H^i(\Mmd{0}{n};\Q)$$
of the de Rham cohomology of $\Mmd{0}{n}$ for all $i$ and $n$. 
Our main result can be expressed
in terms of generating series, as follows. If $X$ is a  smooth
scheme  over $\Q$ of dimension $d$, we will denote its Euler
characteristic  (or rather,  Poincar\'e polynomial)  by:
\begin{equation} \nonumber
e(X) (q)= \sum_{i} (-1)^i \dim_\Q H^i(X;\Q)\, q^{d-i}\ .
\end{equation}
Consider the exponential generating series:
\begin{eqnarray}
g(x)&:=& x- \sum_{n=2}^{\infty}  e(\Mm{0}{n+1})(q^2)\, \frac{x^n}{n!} \ , \nonumber \text{\,} \\
\overline{g}(x)&:=& x + \sum_{n=2}^{\infty}  e(\Mmb{0}{n+1})(q)\, \frac{x^n}{n!} \ . \nonumber  \\
\nonumber
\end{eqnarray}
The following  formula is due to Ginzburg-Kapranov (\cite{GK},
theorem $3.3.2$) and Getzler (\cite{Ge}, \S5.8)
\begin{equation} \label{GK} \overline{g}(g(x)) = g(\overline{g}(x)) = x\ .\end{equation} 
In this paper, we will consider the ordinary generating series:
\begin{eqnarray}
f(x)&:=& x- \sum_{n=2}^{\infty}  e(\Mm{0}{n+1})(q) \,{x^n}  \nonumber\  , \\
f_{\delta}(x)&:=& x + \sum_{n=2}^{\infty}  e(\Mmd{0}{n+1})(q)\, {x^n }  \nonumber \ . \\
\nonumber
\end{eqnarray}

\begin{thm} \label{invthm} The following inversion formula holds:
\begin{equation} \label{mainid} f(f_{\delta}(x)) =
f_{\delta}(f(x))= x\ .\end{equation}
\end{thm}

Using the well-known  formula
\begin{equation} \label{eulopen} e(\Mm{0}{n+1})(q) = \prod_{i=2}^{n-1} (q-i)
\end{equation} 
and the purity of $\Mmd{0}{n+1}$, 
we deduce a recurrence
relation for $e(\Mmd{0}{n+1})$ from ($\ref{mainid})$, and hence also
for the Betti numbers $a_{n,i}$. The proof of equation 
$(\ref{mainid})$ uses the fact that the coefficients in the Lagrange
inversion formula are precisely given by the combinatorics  of
Stasheff polytopes, which in turn determine the structure of the
mixed Tate motive underlying $\Mmd{0}{n}$.

In the  special case $q=0$, the series $f(x)$ reduces to
$x-\sum_{n=2}^\infty (n-1)! \,x^n$, which is essentially the
generating series for the  operad $\mathfrak{Lie}$. Comparing equation 
\eqref{mainid} to Lemma~8 in \cite{ST} we find that the dimensions
$a_{n,n-3}=H^{n-3}(\Mmd{0}{n};\Q)$ are precisely the numbers of
prime generators for $\mathfrak{Lie}$. We expect that there should
be an explicit bijection between $H^{n-3}(\Mmd{0}{n};\Q)$ and the
set of prime generators described in
the proof of Proposition~6 
in \cite{ST}, and, more generally, an operad-theoretic
interpretation of equation
\eqref{mainid} 
for all $q$.

\begin{rmk} The numbers
$a_{n,n-3}$ count the number of  convergent  period integrals  on
the moduli space $\Mm{0}{n}$  defined in \cite{BCS},  called
`cell-zeta values'. Specifically, there is a connected component $X_{\delta}$
of the set of real points $\Mm{0}{n}(\R)\subset \Mmd{0}{n}(\R)$
whose closure $\overline X_\delta \subset \Mmd{0}{n}(\R)$ is a compact
manifold with corners, and is combinatorially  a Stasheff polytope.
For any $\omega \in H^{n-3}(\Mmd{0}{n};\Q)$, one can consider the
integral
$$I(\omega)= \int_{X_\delta} \omega \in \R\ ,$$
which is the period of a framed mixed Tate motive over $\Z$, see 
\cite{G-M}. By a theorem in \cite{Br}, the number $I(\omega)$ is a
$\Q$-linear combination of multiple zeta values. For example, when
$n=5$, we have $a_{5,2} = 1$, and there is essentially a unique such
integral. Identifying $\Mm{0}{5}$ with $\{(t_1,t_2)\in
(\mathbb{P}^1\backslash \{0,1,\infty\})^2, t_1\neq t_2\}$, we can
write $I(\omega)$  as
$$\int_{0<t_1<t_2<1 } \frac{dt_1 dt_2}{(1-t_1)t_2}  = \zeta(2) \ .$$
\end{rmk}
\smallskip

This work was begun at Institut Mittag-Leffler, Sweden, during the
year 2006-2007 on moduli spaces. We thank the institute for the
hospitality.

\section{Geometry of $\Mmd{0}{n}$}
We recall some geometric properties of $\Mmd{0}{n}$ from \cite{Br}.
The set of real points $\Mm{0}{n}(\R)$ is not 
connected but has $n!/2n$ components, and they can be indexed by the set
of dihedral structures\footnote{A dihedral structure on a set $S$ is an
identification of the elements
of $S$ with the edges of an unoriented polygon, i.e.,
considered modulo dihedral symmetries.} $\delta$ on the set $\{1,\ldots,n\}$.
Let $X_{\delta}$
denote one such connected  component. Its closure in the real
moduli space
$$\overline{X}_\delta \subset \Mmb{0}{n}(\R)$$
is a compact manifold with corners. The variety $\Mmd{0}{n}\subset
\Mmb{0}{n}$ is then defined to be the complement
$\Mmb{0}{n}\backslash A_\delta$, where $A_\delta$ is the set of all
irreducible divisors $D\subset \Mmb{0}{n}\backslash \Mm{0}{n}$ which
do not meet the closed cell $\overline{X}_\delta$. Conversely, every
irreducible divisor $D\subset \Mmb{0}{n}\backslash\Mm{0}{n}$ which
does meet the closed cell $\overline{X}_\delta$, defines an
irreducible divisor $D\cap \Mmd{0}{n} \subset\Mmd{0}{n} \backslash
\Mm{0}{n}$. In the case $n=4$, we have: 
$$\Mm{0}{4}\cong \mathbb{P}^1\backslash \{0,1,\infty\} \ , \ \Mmd{0}{4} \cong \mathbb{P}^1\backslash \{\infty\} \ , \ \Mmb{0}{4}\cong \mathbb{P}^1\ ,$$
where $X_\delta$ is the open interval $(0,1)$ and $\overline{X}_\delta$ is the
closed interval $[0,1]$.


In the case $n=5$, one can take four points in general position in $\mathbb{P}^2$ and identify 
$\Mm{0}{5}$ with the complement of a configuration of six lines passing through each pair of points. The compactification $\Mmb{0}{5}$ is obtained by blowing up these four points, giving a total of ten boundary divisors. Picturing $\mathbb{P}^2$ minus the six lines 
one sees that the set of real points $\Mm{0}{5}(\mathbb{R})$ has exactly 12 connected components which are triangles. Choosing one of these components  $X_{\delta}$,  
and blowing up only the two points which meet $X_{\delta}$ yields a space  in which the boundary divisors incident to  $X_{\delta}$  form a pentagon. The space $\Mmd{0}{5}$ is obtained by removing all divisors of $\Mmb{0}{5}$ except the pentagon which bounds $\overline{X}_\delta$.
 Thus we obtain twelve isomorphic varieties $\Mmd{0}{5}$, one for each connected component of $\Mm{0}{5}(\mathbb{R})$.
 
In general, $\overline{X}_{\delta}\subseteq \Mmd{0}{n}(\mathbb{R})$ has the 
combinatorial structure of a Stasheff polytope. Its faces of
codimension $k$ are in bijection
with the set of decompositions 
of a regular $n$-gon into $k+1$
polygons (with at least 3 sides) by $k$ non-intersecting chords.
Suppose,  for each $i\geq 1$, that there are $\lambda_i(D)$ polygons in
a decomposition $D$ 
which has $i+2$ sides.
Then the corresponding face is
$$F_D\cong\prod_{i=1}^{n-2} \prod_{j=1}^{\lambda_i(D)} \overline{X}_{i+2}\ ,$$ 
and $\overline{X}_{i+2}$ has itself the combinatorial structure of a Stasheff
$i$-polytope. 
Note that $\overline{X}_3$ and $\Mm{0}{3}$ are just
points. 
Since a closed polytope is the disjoint union of its open faces, we
deduce the following stratification for $\Mmd{0}{n}$:
\begin{equation}\label{Mstrat}
\Mmd{0}{n}=\coprod_{D} i_D \Big( \prod_{i=1}^{n-2} \prod_{j=1}^{\lambda_i(D)}
\Mm{0}{i+2}\Big)\ .\end{equation} Here, the disjoint
union is taken over all decompositions 
$D$ of a regular $n$-gon, and
$i_D$ is the isomorphism which restricts to the inclusion of each
face $F_D\hookrightarrow  \overline{X}_\delta$. The empty dissection
corresponds to the inclusion of the open stratum $\Mm{0}{n}$.

\begin{exa} There are nine chords in a regular hexagon, six of which
decompose it into a pentagon and trigon, and three of which
decompose it into two tetragons. 
It then has  21 decompositions into
three pieces (a tetragon and two triangles), and 14 into four triangles.  Therefore equation $(\ref{Mstrat})$ can be abbreviated:
\begin{equation}\label{ex6}
\Mmd{0}{6}=\Mm{0}{6}\cup\Big( 6\,\Mm{0}{5}\cup 3\,\Mm{0}{4}^2\Big)
\cup 21 \,\Mm{0}{4} \cup 14\,\Mm{0}{3}\ .\end{equation}
\end{exa}

\section{Purity} 
Since $\Mmd{0}{n}$ is stratified by products of 
varieties $\Mm{0}{r}$, which are isomorphic to an affine complement
of hyperplanes and therefore of Tate type, it follows
that $H^{i}(\Mmd{0}{n})$ defines an element in the category  of
mixed Tate motives over $\Q$. In fact, it was proved in \cite{Br}
that $\Mmd{0}{n}$ is smooth and affine, so it follows by a theorem due to Grothendieck that its cohomology is generated by global regular forms.
Using the   well-known fact that $H^{i}(\Mm{0}{n})$ is pure \cite{Ge}, it follows that the subspace $H^{i}(\Mmd{0}{n})$ is also pure. We
 can therefore work inside the semisimple subcategory (or Grothendieck group) generated by pure Tate
 motives.
We have that, 
\begin{equation}\label{pure}
 H^{i}(\Mmd{0}{n})\cong  \Q(-i)^{a_{n,i}}\ .\end{equation}
The purity of the spaces $\Mmd{0}{n}$ has the important consequence that 
we have an equality of Poincar\'e polynomials (i.e. not only of Euler characteristics), 
\begin{equation}\label{EulerC}
e(\Mmd{0}{n})= \sum_{D} \Big( \prod_{i=1}^{n-2} \prod_{j=1}^{\lambda_i(D)} e(\Mm{0}{i+2})\Big).
\end{equation}

\section{Decompositions of regular $n$-gons}
If $\lambda$ is a partition of a number, we define $\lambda_i$ to be
the number of times $i$ appears in this partition. For each
partition $\lambda$ of $n-2$, we then define $P(\lambda)$ to be the
number of choices of $-1+\sum_i \lambda_i$ non-intersecting
chords 
of an $n$-regular polygon 
that gives rise, for each $i$,
to $\lambda_i$ subpolygons  with $i+2$ sides. 
Thus, $P(\lambda)$ counts the number of decompositions of an $n$-gon of
given combinatorial type. This number is found to be equal to (see
Ex. 2.7.14 on p. 127 in \cite{GJ}): 
\begin{equation}\label{Pform}
P(\lambda)=\frac{(n-2+\sum_i
\lambda_i)!}{(n-1)! \, \prod_i (\lambda_i!)}.\end{equation} 

Combining this result and \eqref{EulerC} we find that, 
\begin{equation} \label{count} e(\Mmd{0}{n})=\sum_{\lambda \vdash n-2} P(\lambda) \cdot \prod_{i=1}^{n-2} e(\mathcal{M}_{0,i+2})^{\lambda_i}. \end{equation}
Using equation \eqref{eulopen} we can now compute the $a_{n,i}$'s for any $i$ and $n$, 

\begin{exa}
From  Example~$(\ref{ex6})$, we have
$$e(\Mmd{0}{6}) = (q-2)(q-3)(q-4) + 6(q-2)(q-3)+3(q-2)^2+21(q-2) +14\ ,$$
which reduces to $q^3+5q-4$. In particular, $a_{6,3}=\dim_\Q
H^3(\Mmd{0}{6},\Q)=4$.
\end{exa}
 Clearly $a_{n,0}=1$ for all
$n$, and it is also easy to see that $a_{n,1}=0$ for all $n$.
In the following table we present the results for $n$ from five to eleven.

\begin{table}[htbp] 
\label{table} \centerline{ \vbox{ \offinterlineskip \hrule
\halign{&\vrule#& \quod \hfil#\hfil \strut \quod \cr
height2pt&\omit&&\omit&&\omit&&\omit&&\omit&&\omit&&\omit&&\omit&&\omit&
\cr & && $a_{n,1}$ && $a_{n,2}$ && $a_{n,3}$ && $a_{n,4}$ && $a_{n,5}$ &&
$a_{n,6}$ && $a_{n,7}$ && $a_{n,8}$ &\cr
height2pt&\omit&&\omit&&\omit&&\omit&&\omit&&\omit&&\omit&&\omit&&\omit&
\cr \noalign{\hrule}
height2pt&\omit&&\omit&&\omit&&\omit&&\omit&&\omit&&\omit&&\omit&&\omit&
\cr &$\Mmd{0}{5}$ && 0 && 1  &&   && &&  && && && &\cr
height2pt&\omit&&\omit&&\omit&&\omit&&\omit&&\omit&&\omit&&\omit&&\omit&
\cr &$\Mmd{0}{6}$ && 0 && 5  && 4  &&  && && && && &\cr
height2pt&\omit&&\omit&&\omit&&\omit&&\omit&&\omit&&\omit&&\omit&&\omit&
\cr &$\Mmd{0}{7}$ && 0 &&  15  && 28  && 22  && && && && &\cr
height2pt&\omit&&\omit&&\omit&&\omit&&\omit&&\omit&&\omit&&\omit&&\omit&
\cr &$\Mmd{0}{8}$ && 0 && 35  && 112  && 206 && 144 && && && & \cr
height2pt&\omit&&\omit&&\omit&&\omit&&\omit&&\omit&&\omit&&\omit&&\omit&
\cr &$\Mmd{0}{9}$ && 0 && 70  && 336  && 1063  && 1704 && 1089 && &&
&\cr
height2pt&\omit&&\omit&&\omit&&\omit&&\omit&&\omit&&\omit&&\omit&&\omit&
\cr &$\Mmd{0}{10}$ && 0 && 126  && 840  && 3999  && 10848 && 15709 &&
9308 && &\cr
height2pt&\omit&&\omit&&\omit&&\omit&&\omit&&\omit&&\omit&&\omit&&\omit&
\cr &$\Mmd{0}{11}$ && 0 && 210  && 1848  && 12255  && 49368 && 119857 &&
159412 && 88562&\cr
height2pt&\omit&&\omit&&\omit&&\omit&&\omit&&\omit&&\omit&&\omit&&\omit&
\cr } \hrule}}
\end{table}

\noindent There are no entries above the diagonal, because
$\Mmd{0}{n}$ is affine. For small $i$, one can use $(\ref{count})$
to write down explicit formulae for $a_{n,i}$ as a function of $n$, e.g.,
$$a_{n,2}= \binom{n-1}{4}\quad \hbox{ and } \quad a_{n,3}=4\binom{n}{6}.$$
Finally, setting $q=0$ in $(\ref{count})$ gives the following closed
formula for the dimension $a_{n,l}$ of the middle-dimensional de
Rham cohomology of $\Mmd{0}{n}$, where $l:=n-3$,
\begin{equation} \label{diagterm}
\dim_\Q H^l(\Mmd{0}{n};\Q)=\sum_{\lambda \vdash l+1}
P(\lambda) \cdot \prod_{i=1}^{l+1}
\Big((-1)^{l}(l+1)!\Big)^{\lambda_i}. 
\end{equation}

\section{An inversion formula}
\begin{proof}[Proof of theorem \ref{invthm}]
The proof is immediate on comparing equation \eqref{count} with the
 combinatorial interpretation  of Lagrange's formula for
the inversion of series in one variable (see \cite{MF}, equation
(4.5.12), p. 412). More precisely, consider the formal power series:
$$u(x) = x-\sum_{i=2}^{\infty} u_i x^i$$
Lagrange's formula states that the formal solution to $v(u(x))=x$ is
given by
$$v(x) = x+\sum_{i=2}^{\infty} v_i x^i $$
where $v_2 = u_2$, $v_3 = 2 u_2^2+u_3$, $v_4 =5u_2^3+5u_2u_3+u_4$, and in general:
$$  v_n = \sum_{\lambda \vdash n-1}  P(\lambda) \cdot \prod_{i=1}^{n-1} u_{i+1}^{\lambda_i}\ , \quad \hbox{ for } n\geq 2\ .$$ 
The theorem follows from $(\ref{count})$ on setting $u_i=e(\Mm{0}{i+1})$.
\end{proof}

\begin{rmk} There is a stratification of $\Mmb{0}{n}$ similar to the one described by \eqref{Mstrat} for 
$\Mmd{0}{n}$, but where $P(\lambda)$ should be replaced by $T(\lambda)$, 
and where $T(\lambda)$ is the number of dual graphs of $n$-pointed stable curves of genus zero 
that has $\lambda_i$ components with a sum of $i+2$ marked points and nodes. 
Now note that from the proof of theorem~\ref{invthm} and \eqref{GK} it follows that 
$T(\lambda)=P(\lambda) \cdot (n-1)!/\prod_i (i+1)!^{\lambda_i}$.
\end{rmk}

In the special case when $q=0$, we deduce the  following corollary.
\begin{cor} The generating series for the dimensions $\dim_\Q
H^{n-3}(\Mmd{0}{n};\Q)$ is obtained by inverting the
series
$$\sum_{n=1}^\infty (n-1)!\, x^n = x+x^2+2 x^3+ 6x^4+\ldots $$
\end{cor}

\begin{rmk}
The cohomology of $\Mm{0}{n}$ is a module over the symmetric group
$\mathfrak{S}_n$ with $n$ elements, whose representation theory can for
instance be found in \cite{Ge} or \cite{K-L}. The  dihedral subgroup $D_{2n}$ which
stablizes a dihedral ordering $\delta$ acts upon the  affine space
$\Mmd{0}{n}$, and hence its cohomology.  It therefore would be
interesting to compute the character of this group action on
$H^\bullet(\Mmd{0}{n})$, and compare its equivariant generating
series to the one obtained by restriction
$\mathrm{Res}^{\mathfrak{S}_n}_{D_{2n}} H^\bullet(\Mm{0}{n})$.
\end{rmk}

\section{A recurrence relation}
Let us alter our series slightly and put $F(x):=-f(-x)$ and $F_{\delta}(x):=-f_{\delta}(-x)$. 
By theorem~\ref{invthm} we find that $F_{\delta} (F(x))=F(F_{\delta}(x))=x$. 
The series $F(x)$ is easily seen to satisfy the differential equation:
$$x^2F'(x) = (F(x) -x)(xq+1)\ .$$
By differentiating  $F_{\delta}(F(x))=x$, we have
$F_{\delta}'(F(x))F'(x)=1$. Substituting the previous expression for $F'(x)$ gives:
$$F_{\delta}'(F(x)) (F(x)-x)(xq+1) = x^2\ .$$
By changing variables $y=F(x)$, where $F_{\delta}(y)=F_{\delta}(F(x))=x$, we obtain:
$$F_{\delta}'(y)(y-F_{\delta}(y))(q\,F_{\delta}(y)+1) = F_{\delta}(y)^2\ .$$
Expanding out gives:
$$yF_{\delta}' -F_{\delta}F_{\delta}'-F_{\delta}^2+qyF_{\delta}F_{\delta}'-qF_{\delta}^2F_{\delta}'=0\ .$$
If we write
$$F_{\delta}(y) = \sum_{n=1}^\infty a_n y^n\ ,$$
then the coefficient of $y^n$ is exactly:
$$na_n - \sum_{k+l=n+1} ka_ka_l -\sum_{k+l=n}a_ka_l+q\sum_{k+l=n}
ka_ka_l -q \sum_{k+l+m=n+1} ka_ka_la_m=0\ .$$ Decomposing the first
sum $\sum_{k+l=n+1} ka_ka_l=(n+1)a_1a_n +\sum_{k=2}^{n-1}
ka_ka_{n+1-k}$, and using the fact that $a_1=1$, gives the
recurrence relation:
$$a_{n} = -\sum_{k+l=n+1,k,l\geq 2} ka_ka_l +\sum_{k+l=n}
(qk-1)a_ka_l -q \sum_{k+l+m=n+1} ka_ka_la_m\ .$$

\begin{thm}
The  recurrence relation
$$a_{n} = -\sum_{\substack{k+l=n+1 \\ k,l\geq 2}} ka_ka_l +\sum_{k+l=n}
(qk-1)a_ka_l -q \sum_{k+l+m=n+1} ka_ka_la_m\ ,$$ with initial
conditions $a_0=0$, $a_1=1$,   has a unique solution
given by
$$a_n=(-1)^{n+1} e(\Mmd{0}{n+1}) .$$

\end{thm}

In the special case $q=0$, we have the following corollary. Note that in theorem~9 of \cite{ST} there is an equivalent presentation of this recurrence relation. 

\begin{cor} The dimensions $b_n:=\dim_\Q
H^{n-2}(\Mmd{0}{n+1};\Q)$ are the unique solutions to the recurrence
relation:
 $$ b_{n} = \sum_{\substack{k+l=n+1\\ k,l\geq 2}} k\,b_kb_l +\sum_{k+l=n}
b_kb_l\ , \hbox{ for} \quad n\geq 2\ ,$$
with initial conditions $b_0=0, b_1=-1$.
\end{cor}
\begin{proof} Set $b_n=-a_n|_{q=0}$ in the previous theorem.
\end{proof}

\bibliographystyle{plain}

\end{document}